\documentclass[11pt,english]{article}
\usepackage{amsmath, amsthm, latexsym, amssymb,url}
\usepackage{amsfonts}
\usepackage{epsfig}

\usepackage{graphicx}
\usepackage{yhmath}
\usepackage{mathdots}
\usepackage{mathtools} 
\usepackage{pifont}

\input xy
\xyoption{all}

\newcommand{\shrinkmargins}[1]{
  \addtolength{\textheight}{#1\topmargin}
  \addtolength{\textheight}{#1\topmargin}
  \addtolength{\textwidth}{#1\oddsidemargin}
  \addtolength{\textwidth}{#1\evensidemargin}
  \addtolength{\topmargin}{-#1\topmargin}
  \addtolength{\oddsidemargin}{-#1\oddsidemargin}
 \addtolength{\evensidemargin}{-#1\evensidemargin}
  }

\shrinkmargins{.6}

\theoremstyle{plain}

\newtheorem{theorem}{Theorem}[section]

\newtheorem{lemma}[theorem]{Lemma}
\newtheorem{proposition}[theorem]{Proposition}

\newtheorem*{teo}{Theorem}

\newtheorem{definition}[theorem]{Definition}

\theoremstyle{remark}
\newtheorem{remark}[theorem]{Remark}

\theoremstyle{definition}
\newtheorem{example}[theorem]{Example}

\def \Z { \mathbb{Z}}
\def \Q { \mathbb{Q}}

\def \R { \mathbb{R}}
\def \det { \text{det}}

\begin{document}

\thispagestyle{empty}
\setcounter{tocdepth}{7}

\title{An introduction to $\Gamma$-number fields.}
\author{Guillermo Mantilla-Soler}

\date{}

\maketitle

\begin{abstract}
It follows from generalities of quadratic forms that the spinor class of the integral trace of a number field determines the signature and the discriminant of the field. In this paper we define a family of number fields, that contains among others all odd degree Galois tame number fields, for which the converse is true. In other words, for a number field $K$ in such family we prove that the spinor class of the integral trace carries no more information about $K$ than the determinant and the signature do.
\end{abstract}

\section{ Introduction}

Let $K$ be a number field, $O_{K}$ its maximal order and let $r_{K}$ and $s_{K}$ be the number of real and complex embeddings respectively. The {\it integral trace form} of $K$ is the isometry class of the $\Z$-quadratic module \[\langle O_{K},  {\rm t}_{K}\rangle\] associated to the trace pairing \[O_{K} \times O_{K} \to \Z;  \ \ \ (x,y) \mapsto {\rm Tr}_{K/\Q}(xy).\]

\noindent Since $(r_{K}+s_{K}, s_{K})={\rm Sign}(\langle O_{K},  {\rm t}_{K}\rangle)$ and ${\rm disc}(K)=\det(\langle O_{K},  {\rm t}_{K}\rangle)$ (see \cite{Ta}) it follows from generalities of quadratic forms (see \cite[IX \S4]{cassels}) that the spinor genus of $\langle O_{K},  {\rm t}_{K}\rangle\ $ determines the signature and the discriminant of $K$. Besides the trivial case of fields of degree less than $3$, there are some interesting instances where the converse of the above holds:

\begin{teo}\cite[Theorem 3.2]{M5} Let $K$ and $L$ be cubic number fields. Then $\langle O_{K},  {\rm t}_{K}\rangle$ and $\langle O_{L},  {\rm t}_{L}\rangle$ belong to the same spinor genus if and only if ${\rm disc}(K)={\rm disc}(L).$
\end{teo}

Recall that for cubic fields the sign of the discriminant determines the signature of the field, hence there is not necessary to mention signatures in the above theorem. There are examples that show that the above theorem can not be extended to isometry class; take for instance (see \cite[Theorem 5.2]{M6} and \cite[Theorem 3.2]{M5}) two non isomorphic cubic fields of positive fundamental discriminant e.g., any two of the four cubic fields of discriminant $32009$. There are some cases in which the isometry class of the trace is determined by the discriminant and the signature:

\begin{teo}\cite[Theorems 4.2, 4.5]{MantiBol} Let $n$ be a positive integer and let $K, L$ be two tame $\Z/n\Z$-number fields. Then \[\langle O_{K},  {\rm t}_{K}\rangle \cong \langle O_{L},  {\rm t}_{L}\rangle \ \mbox{if and only if} \ {\rm disc}(K)={\rm disc}(L).\]
\end{teo}

For odd $n$ there is not need to mention the signature since the fields are totally real. As it turns out, in the above situation, for  $n$ even the equality between discriminants imply the equality of the signatures. The following example shows that the theorem above can not be extended to arbitrary Galois groups, not even abelian ones. All the examples in this paper have been obtained with the help of John Jones' tables of number fields \cite{Jones}. 

\begin{example}\label{KleinExample}

Let $K$ and $L$ be the number fields defined respectively by the polynomials $f_{K}=x^4 - 41x^2 + 144$ and $f_{L}=x^4 - x^3 - 46x^2 - 115x - 35
$. Both fields are quartic $V_{4}$-Galois fields with discriminant equal to $5^2\cdot 13^2 \cdot 17^2$. Since in the first field $p=5$ has one prime factor, while in the second it has two, we see using \cite[Proposition 2.9]{M5} that the integral traces of $K$ and $L$ are not in the same spinor genus, thus they are not isometric.
\end{example}

In this paper we define a class of number fields for which the discriminant and the signature are necessary and sufficient  to determine the spinor class of the integral trace.

\begin{definition}

Let $K$ be a number field, let $p$ be a rational prime and let $e_{1},...e_{g}$ be the ramification indices of $p$ in $K$. The prime $p$ is called {\it $\epsilon$-split homogeneous in $K$} if  $\displaystyle e_{1}=...=e_{g} \ $ for all $1 \leq i \le g.$ 
\end{definition}
The definition of $\epsilon$-split homogeneous is inspired by the behavior of ramification in Galois number fields; however there are non Galois number fields in which every prime is  $\epsilon$-split homogeneous e.g., a field in which every ramified prime is totally ramified. For instance, the number field defined by the polynomial $x^4 - x^3 - 7x^2 + 11x + 3$ is a $S_{4}$ field in which $p=59$ is the only ramified prime and such prime is totally ramified.\\

\begin{definition}\label{Gamma}
A number field $K$ is called {\it $\Gamma$-number field} if the following conditions hold:

\begin{itemize}
    \item[(a)] The field $K$ is tame i.e., every ramified prime $p$ is not wildly ramified.

    \item[(b)] Among all the odd primes $p$ that are ramified in $K$ there is at most one prime that does not satisfy all of the following
    
    \begin{itemize}
    
        \item[$\bullet$]  $p$ is $\epsilon$-split homogeneous in $K$

        \item[$\bullet$] There is an odd number of primes in $K$ lying over $p$ 
        
        \item[$\bullet$] If e is the ramification degree of $p$ in $K$ then, $\displaystyle \frac{[K:\Q]}{e}$ is odd.
    \end{itemize}
    
    \end{itemize}
    
\noindent If there is an odd ramified prime that does not satisfies (b) we call such prime the exceptional prime of $K$.
    
\end{definition}

\begin{remark}
 Notice that the collection of $\Gamma$-number fields contains all tame Galois number fields of odd degree and all tame number fields in which all ramified primes are totally ramified. However, there can be $\Gamma$-number fields that are of even degree, that are not Galois or in which all their ramified primes are not totally ramified. One such example is the sextic field $K$ defined by the polynomial $x^6 - 2x^5 + 3x^4 - 9x^3 + 8x^2 - 7x - 5$. This is a ${\rm D}_{12}$-number field with discriminant $3^3\cdot 23^3$ such that $3O_{K}=\mathcal{P}_{0}^2$ and $23O_{K}=(\mathcal{P}_{1}\mathcal{P}_{2})^2$ for some maximal ideals $\mathcal{P}_{i}$. In all the cases mentioned in this remark no field has an exceptional prime. However, as verified by Example \ref{UnEjemplo}, there are $\Gamma$-fields in which exceptional primes exist. 
 \end{remark}

For $\Gamma$-number fields the signature and the discriminant are enough and sufficient to determine the spinor genus of the integral trace. Moreover, for non totally real fields this can be improved to isometry class. This is the content of our main result:

 \begin{teo}[cf. Theorem \ref{ElPrincipal}] Let $K, L$ be two $\Gamma$-number fields. Suppose that the set formed by exceptional primes of $K$ and $L$ contains at most one element. Then $\langle O_{K},  {\rm t}_{K}\rangle$ and $\langle O_{L},  {\rm t}_{L}\rangle$ belong to the same spinor genus if and only if ${\rm disc}(K)={\rm disc}(L)$ and ${\rm Sign}(\langle O_{K},  {\rm t}_{K}\rangle)={\rm Sign}(\langle O_{L},  {\rm t}_{L}\rangle)$. If $K$ is not totally real, then \[\langle O_{K},  {\rm t}_{K}\rangle \cong \langle O_{L},  {\rm t}_{L}\rangle \ \mbox{if and only if} \ {\rm disc}(K)={\rm disc}(L) \ {\rm and} \  {\rm Sign}(\langle O_{K},  {\rm t}_{K}\rangle)={\rm Sign}(\langle O_{L},  {\rm t}_{L}\rangle)\]
\end{teo}

\begin{example}\label{UnEjemplo}

Consider the sextic fields $K$ and $L$ defined by the polynomials $f_{K}:=x^6 - x^5 - 2x^4 + x^3 + 7x^2 - 6x + 4$ and $f_{L}:=x^6 - 3x^5 + 10x^4 - 15x^3 + 19x^2 - 12x + 3$. Both fields have signature $(r,s)=(0,3)$, discriminant $-3^3\cdot 107^2$ and Galois closure with Galois group ${\rm D}_{12}$. The fields are not isomorphic; in $K$ the prime $p=3$ has only one prime lying over it, while in $L$ it has three. In both fields the prime $p=3$ is $\epsilon$-split with ramification index $e=2$. In particular, $K$ and $L$ are $\Gamma$-fields that are neither Galois or  of odd degree or with totally ramified primes. Moreover, the prime $q=107$ is exceptional in both fields. Since $K$ is not totally real we know by Theorem  \ref{ElPrincipal} that  $\langle O_{K},  {\rm t}_{K}\rangle \cong \langle O_{L}, {\rm t}_{L}\rangle$.

\end{example}

\section{Proofs of our results}

One of the main ingredients we use is the set of $\alpha$-invariants of a number field; for definitions and properties see \cite{M7}. 

\begin{definition}
Let $K$ be a number and let $p$ be a an odd prime. Let $g$ be the number of prime factors of $p$ in $K$ and let  $(e_{1},...,e_{g})$ and $(f_{1},...,f_{g})$ be the ramification and residue degrees of $p$ over $K$. The first ramification invariant of $p$ in $K$ is the integer
\begin{align*}
\alpha_{p}^{K} &:= \left(\prod_{i=1}^{g}  e_{i}^{f_{i}}\right) u_{p}^{(F-g)},
\end{align*} 

where $F=\sum f_{i}$ and $u_{p} \in \{1,...,p-1\}$ is the first non quadratic residue modulo $p$.
\end{definition} 

Among the useful properties of such invariants we have that they determine the genus of the integral trace which in degree at least $3$, thanks to \cite{Manti2}, is the same as the spinor genus. The main relation about $\alpha$-invariants and the integral trace is the following:
\begin{proposition}\cite[Proposition 2.9]{M5}
Let $K,L$ be tame number fields of degree $n \ge 3$. The forms $\langle O_{K},  {\rm t}_{K}\rangle$ and $\langle O_{L},  {\rm t}_{L}\rangle$ belong to the same spinor genus if and only if the following conditions hold:

\begin{itemize}

\item[i)] $\mathrm{disc}(K)=\mathrm{disc}(L)$,

\item[ii)] $s_{K}=s_{L}$,

\item[iii)] For every finite prime $p \neq 2$ that divides the common discriminant of $K$ and $L$ we have that \[ \left( \frac{\alpha_{p}^{K}}{p} \right)=\left( \frac{\alpha_{p}^{L}}{p} \right). \]

\end{itemize}

\end{proposition}

In $\Gamma$-fields the $\alpha_{p}$ invariant of a non-exceptional ramified prime can be expressed, up to squares, in terms of the degree and the discriminant. More explicitly:

\begin{lemma}\label{Lemazo}

Let $K$ be a degree $n$ $\Gamma$-field of discriminant $d$. Let $p\neq 2$  be a ramified prime in $K$ chosen not to be the possible exceptional prime. Then, \[\alpha_{p}^{K} =\frac{n}{n-v_{p}(d)} \mod (\Z_{p}^{*})^{2},\]

where $v_{p}$ denotes the standard $p$-adic valuation in $\Q$.

\end{lemma}

\begin{proof}

Let $p \neq 2$ be a ramified prime in $K$, not exceptional. By definition of $\alpha$ invariants, and since $p$ is $\epsilon$-split homogeneous, $\alpha_{p}^{K}=e^{F}u_{p}^{F-g}$ where $e$ is the ramification invariant of $p$ over $K$. By hypothesis $g$ is odd, and since $n=eF$ we have, by hypothesis as well, that $F$ is odd. Thus, $\alpha_{p}^{K} =e \mod (\Z_{p}^{*})^{2}$. On the other hand, since $p$ is tame, thanks to  \cite[Chapter III, Proposition 13]{Serre}, we have that \[v_{p}(d)=(e-1)F=n-F=n-\frac{n}{e}\] from where the result follows.
\end{proof}

\begin{theorem}\label{ElPrincipal}
 Let $K$and $L$ be two $\Gamma$-number fields. Suppose that the set formed by exceptional primes of $K$ and $L$ contains at most one element. Then $\langle O_{K},  {\rm t}_{K}\rangle$ and $\langle O_{L},  {\rm t}_{L}\rangle$ belong to the same spinor genus if and only if ${\rm disc}(K)={\rm disc}(L)$ and ${\rm Sign}(\langle O_{K},  {\rm t}_{K}\rangle)={\rm Sign}(\langle O_{L},  {\rm t}_{L}\rangle)$. If $K$ is not totally real, then \[\langle O_{K},  {\rm t}_{K}\rangle \cong \langle O_{L},  {\rm t}_{L}\rangle \ \mbox{if and only if} \ {\rm disc}(K)={\rm disc}(L) \ {\rm and} \  {\rm Sign}(\langle O_{K},  {\rm t}_{K}\rangle)={\rm Sign}(\langle O_{L},  {\rm t}_{L}\rangle)\]
\end{theorem}

\begin{proof}
We may assume that $K$ and $L$ have degree at least $3$. We show the non trivial implication. Thanks to Lemma \ref{Lemazo} \[ \left( \frac{\alpha_{p}^{K}}{p} \right)=\left( \frac{\alpha_{p}^{L}}{p} \right) \] for every odd prime $p$ that divides the common discriminant and that is not exceptional. Since $K$ and $L$ have the same signature \[\langle O_{K},  {\rm t}_{K}\rangle \otimes \R \cong \langle O_{L},  {\rm t}_{L}\rangle \otimes \R.\] Since both fields are tame (see \cite[Proposition 2.7]{M5}) \[\langle O_{K},  {\rm t}_{K}\rangle \otimes \Z_{2} = \langle O_{L},  {\rm t}_{L}\rangle \otimes \Z_{2}.\] Moreover, from generalities of $\Z_{p}$-quadratic forms (see \cite[Chp 8, Lemma 3.4]{cassels}) the two forms are isometric over $\Z_{p}$ for every odd not ramified prime $p$. Hence, \[\langle O_{K},  {\rm t}_{K}\rangle \otimes \Z_{p} = \langle O_{L},  {\rm t}_{L}\rangle \otimes \Z_{p}\] with one possible exception; an exceptional prime if it exists. Therefore, by the Hasse-Minkowski principle we have that for every $p$ \[\langle O_{K},  {\rm t}_{K}\rangle \otimes \Q_{p} = \langle O_{L},  {\rm t}_{L}\rangle \otimes \Q_{p}.\] It follows from \cite[Lemma 2.1]{M5} and \cite[Theorem]{M7} that  \[ \left( \frac{\alpha_{p}^{K}}{p} \right)=\left( \frac{\alpha_{p}^{L}}{p} \right) \] for every prime $p$. The result follows from \cite[Proposition 2.9 and Theorem 2.13]{M5}. 

\end{proof}

\noindent
{\footnotesize Guillermo Mantilla-Soler, Department of Mathematics, Universidad Konrad Lorenz,\\
Bogot\'a, Colombia ({\tt gmantelia@gmail.com})}


\begin{thebibliography}{99}


\bibitem{cassels} J.W. S. Cassels, \textit{Rational quadratic forms}, Dover publications, Inc., Mineola, NY, (2008).



\bibitem{MantiBol} W. Bola\~nos, G. Mantilla-Soler, \textit{The trace form over cyclic number fields}, arXiv preprint arXiv: 1904.10080v2  (2019).

\bibitem{Jones} J. Jones, D. Roberts, \textit{A data base of number fields.},  LMS Journal of Computation and Mathematics. {\bf 17, 1} (2014), 595-618.

\bibitem{Manti2} G. Mantilla-Soler, \textit{The Spinor Genus of the integral trace}, Transactions of the American Mathematical Society 369 (2017) , 1547-1577. 


\bibitem{M5} G. Mantilla-Soler. {\em On the arithmetic determination of the trace}, Journal of Algebra (2015) Vol 444, 272-283.

\bibitem{M6} Mantilla-Soler, G., Rivera-Guaca, C. {\em An introduction to Casimir pairings and some arithmetic applications.} arXiv preprint arXiv: 1812.03133v3 (2019).

\bibitem{M7} Mantilla-Soler. {\em The $(\alpha, \beta)$-ramification invariants of a number field.} arXiv preprint arXiv: 1906.04254  (2019).


\bibitem{Serre} J.P. Serre, \textit{Local fields}, Graduate Texts in Mathematics, \textbf{67}. Springer-Verlag, New York-Berlin, 1979. viii+241 pp.





\bibitem{Ta} O. Taussky, \textit{The discriminant matrix of a number field}, J. London. Math. Soc. \textbf{43} (1968), 152-154.

\end{thebibliography}
\end{document}